\documentclass[11pt]{article}
\usepackage{amsmath}
\usepackage{amssymb}
\usepackage{amsthm}
\usepackage{amsfonts}
\usepackage{graphicx}
\usepackage{pdfpages}
\usepackage{subcaption}
\usepackage{xcolor,colortbl}
\usepackage{multicol}
\usepackage{url}

\definecolor{Gray}{gray}{0.92}
\definecolor{LightCyan}{rgb}{0.88,1,1}

\newcolumntype{a}{>{\columncolor{Gray}}c}
\newcolumntype{b}{>{\columncolor{white}}c}

\usepackage{bbm,epsfig,graphics,epic,color,rotating,color}
\numberwithin{equation}{section}

\def\P{{\mathbb P}}
\def\N{{\mathbb N}}
\def\E{{\mathbb E}}

\theoremstyle{plain}
\newtheorem{theorem}{Theorem}

\newtheorem{proposition}{Proposition}

\newtheorem{lem}{Lemma}
\newtheorem{definition}{Definition}
\newtheorem{example}{Example}

\theoremstyle{definition}
\title{Non-Markovian random walks with memory lapses\thanks{Updated version of: On dependent Bernoulli sequences with memory lapses}}
\author{Manuel Gonz\'alez-Navarrete\thanks{Departamento de Estad\'istica, Universidad del B\'io-B\'io, Chile. e-mail: manuelg@ime.usp.br} ,
Rodrigo Lambert \thanks{Faculdade de Matem\'atica, Universidade Federal de Uberl\^andia and Departamento de Matem\'atica, Universidade Federal da Bahia, Brazil. e-mail: lambert@famat.ufu.br}
\date{}
}

\begin{document}
\maketitle

\begin{abstract}
We propose an approach to construct Bernoulli trials $\{X_i, i\ge 1\}$ combining dependence and independence periods, and call it Bernoulli sequence with random dependence (BSRD). The structure of dependence, on the past $S_i = X_1 + \ldots + X_i$,  {defines} a class of non-Markovian random walks of recent interest in the literature. In this paper, the dependence is activated by an auxiliary collection of Bernoulli trials $\{Y_i, i\ge 1\}$, called {\it memory switch sequence}. We introduce the concept of {\it memory lapses property}, which  {is} characterized by intervals of consecutive independent steps in BSRD. The main results include classical limit theorems for a class of linear BSRD. In particular, we obtain a central limit theorem for a class of BSRD which generalizes some previous results in literature. Along the paper, several examples of potential applications are provided.
\end{abstract}


\section{Introduction and motivation}

Independent and identically distributed Bernoulli trials $X_1, \ldots, X_n$ and their related random walk $S_n = X_1 + \ldots + X_n$ are among the most studied subjects in statistics and probability theories. %Many of the properties and theorems about random variable  are treated extensively in several books.
%In this case, it is well known that the random variable $S_n = X_1 + \ldots + X_n$ has a binomial distribution with parameters $n$ and $p$ ($S_n \sim Bin(n,p)$). Many of their properties and theorems are treated extensively in several books of probability theory. %Even more the very first forms of the law of large numbers and central limit theorem were obtained for such $S_n$, by J. Bernoulli and A. de Moivre, respectively, in the XVIII century.
The generalizations of such processes have also been widely investigated in the literature, either by removing the identically distributed hypothesis, or by considering some dependence structure in the sequence.  For more detalis we refer the reader to Feller \cite{Fe}. %In the one hand, if the trials are still independent, but not identically distributed, with the random variables $X_i \sim Bern (p_i)$ for each $1 \le i \le n$, we have the so called Poisson trials, where the distribution of $S_n$ is known as Poisson binomial (see Feller \cite{Fe}).

In this paper we define a sequence $\{X_i,i\geq 1\}$ of Bernoulli random variables, in which each trial has probability of a success either as function of the number of previous successes $S_n$ or independent of that. The dependence will be activated/inactivated by a latent collection of independent Bernoulli trials $\{Y_i, i\ge 1\}$, called the {\it memory switch sequence}.
In other words, we construct Bernoulli sequences which have flexibility to combine dependent and independent periods. The dependence will be in ``on'' (resp. ``off'') mode, whenever the switch factor ``$Y_i=1$'' (resp. ``$Y_i=0$''). We call the model as Bernoulli sequences with random dependence (BSRD).

The dependence structure is taken from classical applications in epidemiological studies (see, for instance, Chapter 7 of Zelterman \cite{Ze}). However, as noted in recent years, a class of non-Markovian random walks \cite{BB,HK,Ku,ST} can also be defined by that kind of dependence structure. Formally, let $\mathcal{F}_i = \sigma(X_1, \ldots, X_i)$ be the $\sigma$-field generated by the sequence $X_1, \ldots, X_i$. The probability of a success on the $(i+1)$-th trial, given its past $\mathcal{F}_i$ satisfies $\P(X_{i+1} = 1 | \mathcal{F}_i) = \P(X_{i+1} = 1 | S_i)$. In words: the whole information of the past is summarized in  $S_i$. In particular, for these processes, the present paper adopts the following notation
\begin{equation}
\label{Yu2}
P^*_i(s) = \P(X_{i+1} = 1 | S_i = s),
\end{equation}
for $i \ge 1$ and $0\le s \le i$. The BSRD proposes another way to define the conditional probabilities of $X_{i+1}$ given $S_i$.

We characterize the BSRD and the property of {\it memory lapses} (see Definition \ref{defi}), which is given by a string of $0$'s in $\{Y_i,i\geq 1\}$ that represents a period of independence in the steps of the BSRD $\{X_i,i\geq 1\}$. The main results are classical limit theorems for $S_n$ in a class of BSRD with linear dependence. More specifically, we show conditions on the parameters to obtain strong law of large numbers, central limit theorem and an invariance principle for $S_n$. We also generalize asymptotic results for some models in literature, providing explicit limit quantities related to the asymptotic distributions.

The paper is organized as follows. In Section \ref{sec:model} we define the BSRD and explain its relation with a family of non-Markovian random walks. Section \ref{sec:thm} includes the main results. Section \ref{sec:app} provides examples and discusses applications and mathematical properties of the memory lapses. Finally, the tools and proofs of main theorems and other results are given in Section \ref{sec:proof}.

\bigskip

\section{Bernoulli sequences with random dependence}
\label{sec:model}

Let $\{X_i,  i \geq 1\}$ be a sequence of Bernoulli trials and $\{Y_i, i \geq 1\}$ an auxiliary collection of independent Bernoulli random variables, with $\P(Y_i=1) =\lambda_i$ and independent of $S_i= X_1 + \ldots + X_i$, for all $ i \ge 1$. The dependence structure of $\{X_i\}$ (for short notation) will be associated to the sequence $\{Y_i\}$, which we call {\it memory switch sequence}. Formally, the probability of having a success at $(i+1)$-th trial is related to its past information $S_i$ and the realization of the random variable $Y_i$, by $\P(X_{i+1} = 1 | \mathcal{F}^X_i, \mathcal{F}^Y_i) = \P(X_{i+1} = 1 | S_i, Y_i)$, where $\mathcal{F}^X_i = \sigma(X_1, \ldots, X_i)$ is the $\sigma$-field generated by the sequence $\{X_i\}$, and similar for $\mathcal{F}^Y_i$. Let denote
\begin{equation}
\label{ourProb}
P_i(s,y) = \P(X_{i+1} = 1 | S_i = s, Y_i = y),
\end{equation}
for all $i\ge 1$, $0\le s \le i$ and $y=0,1$. We recall that $P$ is defined in a different probability space than $P^*$ in \eqref{Yu2}. We remark that, for $y=0$ the probability $P_i(s,0)$ will not depend on the previous successes $S_i$. However, if $y=1$, the dependence exists. We think random variable $Y_i$ as a latent factor that determines the choice of dependence (or independence) on the past $S_i$ for the trial $X_{i+1}$. Now, it is provided the definition for the so-called Bernoulli sequence with random dependence (BSRD).

\begin{definition}
\label{def1}
A BSRD is the collection $\{X_i\}$, with memory switch sequence $\{Y_i\}$, defined by conditional probabilities \eqref{ourProb} satisfying
\begin{equation}
\label{condi}
P_i(s,0) = P^*_i(0) \ \ \ \text{and}  \ \ \ P_i(s,1) = P_i^*(s),
\end{equation}
where $P^*_i(\cdot)$, given by \eqref{Yu2}, is the probability of an embedded dependent Bernoulli sequence.
\end{definition}

We propose the formulation of some correlated random walks and related processes, as BSRD in Definition \ref{def1}. Based in the applications exposed in \cite{HK}, the idea is to include a factor $\{Y_i\}$, maybe genetic, linguistic or economical that will provide some interference in the dependence structure. In this sense, the Bernoulli random variables mean steps up or down, if $X_i=1$ or $X_i=0$, respectively.

A first model to be analysed was studied by Hod and Keshet \cite{HK}, and its representation by conditional probabilities \eqref{Yu2} is given by 
$$
P^*_i(s) = \dfrac{1}{2}\left( 1 - \mu\dfrac{i}{i+l}\right) + \mu \dfrac{s}{i+l} \ ,$$
where $-1 < \mu < 1$ and $l >0$ is a \emph{constant transient time}, that plays the following correlation rule: For $i<<l$, the \emph{past effect} is not too intense, while for $i>>l$ it does. In \cite{HK} the authors also discuss several applications.

The second case was called elephant random walk (ERW) and has been introduced in \cite{ST}. In the ERW it is supposed that the elephant remembers its full history and chooses its next step as follows. First, it selects randomly a step from the past, and then, with probability $p \in [0,1]$, it repeats what it did at the remembered time, whereas with the complementary probability $1 - p$, it makes a step in the opposite direction. In this respect, the probability of $i$-step of ERW in a BSRD version is given by
\begin{equation}
\label{ERW}
P_i(s,y) = (1-p) + (2p-1) y \dfrac{s}{i},
\end{equation}
where $p \in [0,1]$ and $\{Y_i\}$ are i.i.d. with parameter $\lambda \in [0.1]$. For the BSRD in \eqref{ERW}, a memory lapse will be an interval of consecutive steps for which the elephant does not look to the history of steps. That is, it makes the step independently of its past.

As noted in \cite{BB}, the original ERW (that is $\lambda=1$ in \eqref{ERW}) can be represented by a generalized Polya-type urn. In fact, classical asymptotic results can be obtained by using different mathematical approaches (for details, see \cite{Ber,Ku}).

In a different framework, a generalized binomial distribution was proposed by Drezner and Farnum \cite{DF}. It is given by
\begin{equation}
\label{epi0}
P^*_i(s)= (1-\beta)p + \beta \frac{s}{i},
\end{equation}
where $0\le \beta \le 1$ and $p=\P(X_1=1)$ is the initial probability. In \cite{BHW} it was studied a model for which each $i$-th individual makes up his mind about whether to adopt decision A ($X_i=1$) or B ($X_i=0$). The model is related to Example \ref{ex3} below and can be defined by same conditional probabilities \eqref{epi0}. The formulation was given as a generalization of the Polya urn, that is, if $\beta=1$ the conditional probabilities \eqref{epi0} define a Polya urn process.

%The work \cite{BHW} did not link their formulation with the results obtained in Heyde \cite{He} for the generalized binomial distribution introduced in \cite{DF}. That is, $S_n$ converges almost surely to $p$, as $n \to \infty$, and a central limit theorem for $S_n$ with variance going to $\frac{p(1-p)}{1-2\beta}$.

In this sense, in Gonz\'alez-Navarrete and Lambert \cite{GL} it was introduced an urn process with the property of memory lapses. The relation between dependent Bernoulli sequences and P\'olya-type urn processes is discussed therein. We refer the book of Mahmoud \cite{Mah} for details about the theory of P\'olya urns, to Baur and Bertoin \cite{BB} for its connection with the ERW, and to Janson \cite{Janson} for a general theory about convergence of P\'olya urn models.

{Let us now state a simple technical result that allows us to obtain finite-dimensional laws $\P(X_1, \cdots ,X_n)$, for all $n \geq 1$ and then check the existence of the BSRD as a stochastic process. Consider a BSRD $\{X_i, 1\le i \le n\}$, as in Definition \ref{def1}. Using independence of $Y_i$ and $S_i$ combined with definition \ref{ourProb} we have}

\begin{equation}
\label{propo1}
\begin{array}{ll}
P_i(s):= \P(X_{i+1} = 1 | S_i = s) &= \displaystyle\sum_{y \in \{0,1\}}  \P(X_{i+1} = 1 | S_i = s, Y_i = y) \P(Y_i=y)\\
&= \displaystyle\sum_{y \in \{0,1\}} P_i(s,y) \P(Y_i=y)\\
& = (1-\lambda_i)P^*_i(0) + \lambda_i P^*_i(s).
\end{array}
\end{equation}

In other words, if $\lambda_i \equiv 0$ ($\lambda_i = 0$ for all $i\ge 1$), the BSRD has the same probability measure than independent Bernoulli trials, with probability of success $P^*_i(0)$. In particular, if $P^*_i(0)=1/2$, we obtain the well-known simple symmetric random walk. Instead, if $\lambda_i \equiv 1$, the BSRD correspond to the embedded dependent trials defined in \eqref{Yu2}. 

Now, let us take a look at BSRD as being stochastic processes. Particularly, we remark two important facts related to a realization of such a process (see \eqref{our2} as reference).

\begin{itemize}
\item [(R1)] If $Y_i=0$ then $X_{i+1}$ is independent of its past. In other words, $X_{i+1}$ is chosen from a Bernoulli distribution with parameter $\alpha_i = P^*_i(0)$ independent of the observation from $(X_1, \ldots, X_i)$.
\item [(R2)] If $Y_i=1$ then $X_{i+1}$ should be dependent of its past. The probability $P^*_i(s)$ is defined by the embedded dependent Bernoulli sequence.
\end{itemize}

Note that, by (R1) a period of independence in the BSRD is given by a string of $0$'s in the memory switch sequence. This fact leads us to the following definition

\begin{definition}
\label{defi}
A {\it memory lapse} in the BSRD $\{X_i\}$ is an interval $I\subset \N$ such that $Y_i=0$ for all $i\in I$ and there is no interval $J \neq I$ with $J \supset I$ such that $Y_i=0$ for all $i \in J$. The length of the lapse is $|I|$.
\end{definition}

In some sense, we think this period as a lapse because after these, the model always will recover the dependence on the whole past, given by the conditional probability $P^*_i(s)$. More details about the notion of memory lapses, their applications and some mathematical properties will be discussed in Section \ref{sec:app}.

\section{Main results}\label{sec:thm}

We study, as embedded dependent processes, the family of linear-dependent Bernoulli sequences described in Wu et al. \cite{WQY}. This leads us to an application of BSRD in a wide class of previous models in the literature. Using notation \eqref{ourProb} we provide the following definition
\begin{equation}
\label{our2}
\begin{array}{c}
P_i(s,y) = \displaystyle\alpha_i + \beta_i \frac{ s y }{i} \ \ \ \text{and} \ \ \ \P(X_1=1)=P_0(s)=\alpha_0,
\end{array}
\end{equation}
for all $i \ge 1$, $0\le s \le i$ and $y=0,1$. The sequences $\{\alpha_i\}$ and $\{\beta_i\}$ of parameters must satisfy the following conditions: $\alpha_i, \beta_i \ge0$ and $\alpha_i + \beta_i \le 1$. We remark that $\beta_i$ measures the strength of the dependence, while parameter $\lambda_i = \P(Y_i = 1)$ represents the probability to that dependence actually exists.

In the following results we show conditions to obtain classical limit theorems for $S_n$ in the BSRD defined in \eqref{our2}. %Then, let denote
Let first define some quantities appearing in the main results. We start with $a_1=1$, and for $n \geq 2$
\begin{equation}
\label{an}
a_n = \prod_{k=1}^{n-1}\left(1 + \dfrac{\beta_k}{k} \lambda_k\right) \ , \ \ \ A_n^2 =\displaystyle\sum_{i=1}^n \frac{1}{a_i^2} \ \ \text{ and } \ \  B_n^2 =\displaystyle\sum_{i=1}^n \frac{p_i(1-p_i)}{a_i^2},
\end{equation}
where $p_i=\P(X_i=1)$. Now we state the main theorems of this paper. The first one provides a strong law of large numbers (SLLN) for $S_n$.

\begin{theorem}\label{teo1}
Consider a BSRD $\{X_i\}_{1\le i\le n}$ as in \eqref{our2}.
Then,

\begin{equation}
\label{LLN}
\displaystyle\lim_{n\to \infty} \frac{S_n - \E(S_n)}{n} = 0 \ \ \ \text{a.s}
\end{equation}
if and only if,
\begin{equation}
\label{cond1}
\displaystyle\sum_{k=1}^{\infty} \frac{1-\beta_k\lambda_k}{1+k} = \infty.
\end{equation}

\end{theorem}
 
In \eqref{cond1} we recall that the relation between the real number sequences $\{\beta_i\}_{i \in \mathbb{N}}$ and $\{\lambda_i\}_{i \in \mathbb{N}}$ plays a central role in the SLLN. For instance, if $\beta_i = 1/\lambda_i$ for all $i$, we do not get the convergence \eqref{LLN}. A more accurated discussion about the relation between these parameters and its consequences is provided at \cite{GL} in the context of Generalized P\'olya-type urns.

In what follows, we present the second theorem. It provides an invariance principle for $S_n$.
  \begin{theorem}\label{teo2}
Suppose in addition to the hypothesis of Theorem \ref{teo1} that $\lim_{n \to \infty} B_n = \infty$ and \break $\limsup_{n \to \infty} A_n/B_n < \infty$. Then it is possible to redefine $\{X_i \}$ in a new probability space without changing its distribution and there exists a standard Brownian motion $\{W(t)\}$ defined on the same probability space such that
\begin{equation}\label{IP}
\mbox{(a)} \ \ \dfrac{\left|\frac{S_n - \E(S_n)}{a_n} - W(B_n^2)\right|}{B_n\sqrt{\log\log B_n}} \xrightarrow{a.s.} 0 \ \  \ \ \ ; \ \ \ \mbox{(b)} \ \ \dfrac{\left|\frac{S_n - \E(S_n)}{a_n} - W(B_n^2)\right|}{B_n} \xrightarrow{\P} 0 \ .
\end{equation}

\end{theorem}

In \eqref{IP}, the symbols $\xrightarrow{a.s.}$ and $\xrightarrow{\P}$ mean almost-sure convergence and convergence in probability, respectivelly, with all limits taken as $n$ diverges.

We recall that the central limit theorem (CLT) and also the law of the iterated logarithm (LIL) for BSRD follow straightforward from Theorem \ref{teo2} and the CLT and LIL for the standard Brownian motion. 

Finally, we generalize some asymptotic results for previous models in the literature. In particular, the generalized binomial proposed in \cite{DF}, as expressed in \eqref{epi0}. The following result provides explicit limiting proportion of successes for this particular class of BSRD.

\begin{theorem}\label{teo3}
Let $\{X_i\}$ a BSRD with $\{Y_i\}$ i.i.d. parameter $\lambda \in [0.1]$. The conditional probabilities \eqref{our2} are given by
\begin{equation}
\label{epi}
P_i(s,y)=(1-\beta) \alpha_0 + \beta y\dfrac{s}{i},
\end{equation}
where $0\le \beta \le 1$ and $\alpha_0=\P(X_1=1)$

\begin{itemize}
\item[(i)] If $\beta \lambda < 1/2$ then
\begin{equation}\label{tcl1}
\dfrac{1}{\sqrt{n}}\left({S_n - \frac{n\alpha_0 (1 - \beta)}{1-\beta \lambda}}\right) \xrightarrow{d} N(0,\sigma^2) \ \ ; \ \ \mbox{as} \ n \to \infty \ ,
\end{equation}
where 
\begin{equation}\label{var1}
\sigma^2 = \dfrac{\alpha_0(1-\alpha_0-\beta(\lambda-\alpha_0))(1-\beta)}{(1-2\beta \lambda)(1-\beta \lambda)^2}
\end{equation}
\item[(ii)] If $\beta \lambda = 1/2$ then
\begin{equation}\label{tcl2}
\dfrac{S_n - {2n\alpha_0 (1 - \beta)}}{\sqrt{n\log n}} \xrightarrow{d} N(0,\sigma^2) \ \ ; \ \ \mbox{as} \ n \to \infty \ ,
\end{equation}
where 
\begin{equation}\label{var2}
\sigma^2 = 4\alpha_0\left(\frac{1}{2}-\alpha_0(1-\beta)\right) (1 - \beta ) \ .
\end{equation}
\end{itemize}

\end{theorem}

 The proofs of all theorems will be given in Section \ref{sec:proofTh} and are based on classical results for convergence of bounded martingale differences (Theorems \ref{teo1} and \ref{teo2}), and its relation with a class of generalized P\'olya urns (Theorem \ref{teo3}).

\bigskip

\section{The memory lapses property}
\label{sec:app}

This section presents various examples and discusses potential application for the BSRD. The notion of memory lapses is exploited in the context of each example/application. The goal is to discuss the flexibility provided by this property in further mathematical modelling of real problems. 

Consider a contagious disease affecting a finite population. Each of the individuals can be described by Bernoulli distributed indicators of their disease status. In other words, $X_i =1$ if the $i$-th member is a diseased case and $X_i$ is zero if this member is otherwise healthy. The collection of Bernoulli trials $\{X_i, i\ge 1\}$ has been extensively studied in literature by conditional probabilities \eqref{Yu2}. {We refer} for instance \cite{Ze}. 

In the approach of BSRD, we suppose that a genetic or environmental factor can be represented by the memory switch sequence $\{Y_i\}$. That is, such factor activates/inactivates the dependence structure of the model. Then, for instance, consider the model in \eqref{epi} and the following two situations.

\begin{example}\label{ex1}
Let denote switch factor $Y_i= \mathbb{I}_{\{Z_i < z\}}$, where $Z_i$ is a continuous random variable which represents the concentration of an ``immunizing antibody'' related to $(i+1)$-th individual, and $z$ is the critical value for immunization. Then, if $Z_i \ge z$, the probability of $(i+1)$-th member become diseased is independent of the historical of such disease in the population. Otherwise, the probability will increase by the (familiar) historical of successes $S_i$, the number of sick individuals until  {time} $i$.
\end{example}

\begin{example}\label{ex2}
In other situations, $Y_i$ may denote  {a vaccine} to prevent infection. Where $Y_i=0$ means that the vaccine in effective when applied on the $(i+1)$-th member of the population. Otherwise, the vaccine does not actuate and the probability increases by the historical of  {sick} individuals.
\end{example}

Therefore, as in Definition \ref{defi}, a memory lapse of length $l$ could be interpreted as a group of $l$ consecutively patients for which the vaccine is effective. That is, the historical of sickness is neglected. 

In what follows we give other examples trying to understand better the memory lapses property in a bit different situations.

\begin{example}\label{ex3}
Customers can buy a product $(X_i=1)$ by necessity or by another reason. Thinking on it, an advertising program $\{Y_i\}$ is launched, that can or cannot reach the customers. Consider the conditional probabilities \eqref{our2}. Let be $\alpha_i$ the quantity related to necessity and $\beta_i$ the quantity related to the ``social pressure" (history of selling) of the product as defined in \cite{BHW}, associated to $(i+1)$-th customer. 
Then, if $Y_i=1$, that is, the publicity reached him/her, the ``social pressure" will raise the probability of the customer buying the product (see remark (R2) in Section \ref{sec:model}).
\end{example} 

\begin{example}\label{ex4}
Imagine that there is a robot (computer) used to buy/sell some  {asset} in the stock market. It obeys an algorithm that determines such decision. At the $(i+1)$-th decision, the robot tries to access the historical buying/selling activity. If it has success ($Y_i=1$), then its decision will be based on the history of the asset. If it doesn't, due by some noise or stoppage, for instance, then it decides to buy (or sell) the asset only by flipping a coin.
Here we are supposing that although its memory is inaccessible ($Y_i=0$), it has to take a decision anyway. 
\end{example} 
%
% \1{
\begin{example}\label{hot} (The hot hand in basketball) In basketball, there is a common belief that the probability of hitting a shot after a hit is greater than the probability of hittting after a miss. Moreover, a sequence of consecutive hits will increase the probability of a hit in the next shot. This is the so-caled "hot hand phenomena" (we refer the reader to the paper from Gilovich et al \cite{GVT}). In the context of BSRD, the hot hand can be regarded as a consecutive string of 1's in $\{X_i\}$ and $\{Y_i\}$ (or a string without misses and without memory lapses). Then the probability of hitting will raise since the basketball player comes from a success history.
\end{example}

Therefore, in Example \ref{ex3} a memory lapse is denoted by consecutive costumers for which the publicity does not reach them. In a similar sense, the memory lapses can be included in the model on social behaviour about soft technologies introduced by Bendor et al. \cite{BHW}. In the case of Example \ref{ex4} the memory lapses are periods of decisions taken without looking to the historical of selling of the  {asset}, described by interruptions in the system or the intranet connection. In view of Example \ref{hot} the memory lapse can be though as a period in which the opposite team moves the defence to the player, which is a common strategy when the objective is to stop some player in hot hand.

Finally we mention the work of Businger \cite{SB}, which recently used similar ideas to construct the so-called shark random swim.

\subsection{Characterization of switch sequence}

In a complementary line, we obtain some mathematical results about the occurrence of memory lapses in a BSRD. In particular, we analyse the property of memory lapses by studying random variables defined on $\{Y_i\}$.

First of all, note that in previous sections we supposed the trials $\{Y_i\}$ being independent, with $\P(Y_i = 1) = \lambda_i$. We highlight here that the problem to study patterns in an independent Bernoulli sequence (Poisson trials) has been widely  {studied} in literature \cite{Fe}. Therefore there are several possibilities to analyze such patterns by assuming different forms for the parameters collection $\{\lambda_i\}$. The assumptions about that parameters are given by specific application of the BSRD. However, as an illustration, we focus in two particular cases and compare their behaviours. The results are obtained by using techniques from recent works in the study of pattern strings in Bernoulli sequences (see, for instance \cite{Ho2,Ho,HoK,SS2}).

In a first part we study waiting times for switch sequence $\{Y_i\}$ based on the approach introduced in \cite{EbSo}. In this sense, we focus in the so called \emph{frequency (FQ) and succession (SQ) quotas} for Bernoulli trials. Revisiting the situations given above, we think applications to the analysis of these FQ and SQ problems.

Formally, in the case of FQ, we denote $W^{FQ}_r$ the waiting time until $r$ failures (or $0$'s) has been observed in $\{Y_i\}$. We assume the memory switch sequence being i.i.d. We recall that since we are counting how many independent Bernoulli trials until observe $r$ failures. Therefore the random variable $W^{FQ}_r$ follows a negative binomial distribution with parameters $1-\lambda$ and $r$.

\begin{example}
\label{ex:fq}
(Relation between FQ and Example \ref{ex3}) Suppose that there is a criterion to evaluate publicity. For instance, if it completes a fixed quantity $r$ of non-reached customers, then the publicity is removed. That is, we are  {interested on} the waiting time until attain a given quota of 0's. As we said above, decision of remove publicity should  {regard} the probability distribution of a negative binomial random variable
\end{example}

Note also, if we want to see the proportion of failures, then we will deal with a binomial distribution. To illustrate this,  {suppose that} we choose a ``critical proportion'' of failures $\delta \in (0,1)$. Then the probability of reach this proportion will be given by $P(Z_n \leq \lfloor n \delta \rfloor)$, where $Z_n \sim Bin(n,1-\lambda)$. Here we recall that by $\lfloor x \rfloor$ we mean the greatest integer less or equal than $x$.

On the other hand, we consider a succession quota (SQ) problem.  Let $W^{SQ}_s$ be the waiting time until  {the first memory lapse of size $s$ is observed}. It is possible to obtain a probability generating function for $W^{SQ}_s$, given by
\begin{equation}
\label{SQ}
\phi_W(t) = \dfrac{(1- qt)(q t)^s}{(1-\lambda t) (1-qt) - \lambda qt^2(1-(q t)^{s-1})},
\end{equation}
where $q=1-\lambda$ (see \cite{EbSo}). In particular, we obtain $\E(W^{SQ}_s) = (1-q^s) / (\lambda q^s)$.

However, as other situations from SQ, in next example we do not look at the waiting time of a memory lapse. Instead, we analyse a pattern of $1$'s, that gives us a particular information about the situation under study. Of course, if one interchanges $\lambda$ by $q$ in \eqref{SQ} we obtain the corresponding probability generating function.

\begin{example}
\label{ex:sq}
(Relation between SQ and Example \ref{ex2}) We define the waiting time given by the first instant that we obtain $s$ consecutive 1's in the memory switch sequence. In other words, we are  {interested} in the first time for which the vaccine {has not effects in} $s$ consecutive patients.
\end{example}

Let us now address a second kind of questions. We modify the approach to study pattern behaviours of the {\it memory lapses}. Then, define the random variables $M_l(n)$, $l \in \{1,2,\ldots, n\}$, representing the number of memory lapses of length $l$ in the first $n$ trials of $\{Y_i\}$. Formally,
\begin{equation}
\label{Mnr}
\begin{array}{ll}
M_l(n) = \displaystyle\sum_{k=1}^{n-l-1} & Y_k (1-Y_{k+1}) \cdots (1-Y_{k+l}) Y_{k+l+1}\\
& + \ (1-Y_1) \cdots (1-Y_{l}) Y_{l+1} \ + \ Y_{n-l}(1-Y_{n-l+1}) \cdots (1-Y_n),
\end{array}
\end{equation}
for $n>l$, and where second and third terms represent a memory lapse at the beginning and at the end of the sequence, respectively. In other words, $M_l(n)$ is the number of runs of 0's (see \cite{Ho2,HoK}) of length $l$ in the first $n$ trials of $\{Y_i\}$.

If we are able to obtain information about random variables given in \eqref{Mnr} we could say something about, for instance, the groups of consecutive customers for which the publicity does not reach, in Example \ref{ex3}. Also with $M_l(n)$ we count the strings of $l$ consecutive decisions without looking for the data in Example \ref{ex4}.

In addition to \eqref{Mnr}, we will be interested in periods of alternate dependence of BSRD. Formally, as functions of sequence $\{Y_i\}$, we have
\begin{equation}
\label{Anr}
A_l(n) = \displaystyle\sum_{k=1}^{n-2l+1} Y_k (1-Y_{k+1}) \cdots Y_{k+2l-2} (1-Y_{k+2l-1}),
\end{equation}
being the number of alternating dependence-independence periods of size $2l$ in the first $n$ trials of the BSRD $\{X_i\}$, where $n> 2l$. These random variables could give us information about periods of high interference in Example \ref{ex4}. Of course their characterization complements the informations given by \eqref{Mnr}.

In the next result we obtain expectation of random variables $M_l(n)$ and $A_l(n)$. We consider the assumptions i.i.d. and the collection $\lambda_i = a / (a+b+i-1)$ for $a>0$, $b\ge 0$ and $i\ge 1$, usually denoted by $Bern(a,b)$ (see \cite{Ho,SS2}).

\begin{proposition}
\label{propo2}
\begin{itemize}
\item [i)] If $\{Y_n\}$ are i.i.d. with parameter $\lambda$, then
\begin{equation}
\label{M1}
\E(M_l(n) )= (1-\lambda)^l [ 2 \lambda + \lambda^2(n-l-1)], \ \ l < n,
\end{equation}
and
\begin{equation}
\label{A1}
  \E(A_l(n) ) = (n-2l)( \lambda(1-\lambda))^l, \ \ 2l<n.
\end{equation}

\item [ii)] If $\{Y_n\}$ are $ {Bern}(1,b)$, then
\begin{equation}
\label{M2}
\E(M_l(n) )= \dfrac{2b+l}{(b+l)(b+l+1)}, \ \ 1 \le l < n,
\end{equation}
and
\begin{equation}
\label{A2}
  \E(A_l(n) ) = \displaystyle\sum_{k=1}^{n-2l+1} \displaystyle\prod_{i=1}^{l}\dfrac{1}{k+b+2i-1} \ \  2l < n.
\end{equation}
\end{itemize}
\end{proposition}

We remark that it is possible to obtain second moments as recursive functions of expectations in Proposition \ref{propo2}. Note that \eqref{M2} does not depend on the value of $n$, this is an interesting feature of Poisson trials $Bern(a,b)$, which have been studied in different contexts. For instance, the sequence $Bern(1,0)$ arises in the limit in the study of cycles in random permutations and record values of continuous random variables. Moreover, the sequence $Bern(a,0)$ has some applications in nonparametric Bayesian inference and species allocation models (see \cite{Ho,SS2} and references therein).

Finally, we remark that, if $\sum_{k\ge 1} \E(Y_k) = \infty$, by second Borel-Cantelli lemma we have that, $\sum_{k\ge 1} Y_k = \infty$ almost sure. This implies that, with probability one, the successes occur infinitely often. However, if $\{Y_i\}$ are $Bern(1,b)$, the number of strings $\{11\}$ is almost surely finite, since $\sum_{k\ge 1} \E(Y_k Y_{k+1}) < \infty$.

\bigskip

\section{Proofs}
\label{sec:proof}

\begin{proof}[Proof of Proposition \ref{propo2}]
Note that by independence we have from equation \eqref{Mnr}
\begin{equation}
\begin{array}{ll}
\E(M_l(n)) =&  \displaystyle\sum_{k=1}^{n-l-1}  \lambda_k (1-\lambda_{k+1}) \cdots (1-\lambda_{k+l}) \lambda_{k+l+1} \\[0.4cm]
  &  + \ (1-\lambda_1) \cdots (1-\lambda_{l}) \lambda_{l+1} + \ \lambda_{n-l}(1-\lambda_{n-l+1}) \cdots (1-\lambda_n).
 \end{array}
\end{equation} 
 
 In the part i), the case $\{Y_n\}$ are iid, by a straightforward calculation we obtain \eqref{M1}. Similarly, we prove \eqref{A1}.
 
Now, in part ii), if $\{Y_n\}$ are $Bern(1,b)$, we should calculate
\begin{equation}
\E(M_l(n)) =  \displaystyle\sum_{k=1}^{n-l-1}  \left( \dfrac{1}{k+b+l} - \dfrac{1}{k+b+l+1} \right) + \dfrac{b}{(b+l)(b+l+1)} + \dfrac{1}{n+b}.
\end{equation}
Then, solving the telescoping sum, we obtain \eqref{M2}. By similar arguments we get $\E(A_l(n))$.

\end{proof}

\subsection*{Proof of main results}
\label{sec:proofTh}

The proof of Theorem \ref{teo1} is based in general results about convergence of martingale differences (see \cite{H&H}). First, let
\begin{equation}\label{mart}
M_n = \dfrac{S_n - \E(S_n)}{a_n}
\end{equation}
where $a_n$ are given by \eqref{an}, and denote
\begin{equation}
\label{martdiff}
D_1 = M_1 \ \ \ \ \  D_n = M_n - M_{n-1}, \ \ n \ge 2.
\end{equation}

We aim to prove that $\{D_n, \mathcal{F}_n, n \ge 1\}$ is a sequence of bounded martingale differences. First, we need particular cases of Theorem 2.17 and Corollary 3.1 included in Hall and Heyde \cite{H&H}. We state the results without proof.

\begin{lem}
\label{lemA}
Let $\{Z_n,\mathcal{F}_n, n \ge 1\}$ be a sequence of martingale differences. If 
$\sum_{n=1}^{\infty} \E[Z_n^2| \mathcal{F}_{n-1}] < \infty$ a.s., then $\sum_{i=1}^{n}Z_i$ converges almost surely.
\end{lem}

%
%
%\item [(b)] If there exists a sequence of positive constants $\{B_n\}$ such that $B_n \to \infty$ as $n \to \infty$ and $\frac{1}{B^2_n} \sum_{j=1}^n \E[Z_j^2| \mathcal{F}_{j-1}] \xrightarrow{p} \sigma^2$. Then $\frac{\sum_{j=1}^n Z_j}{B_n} \xrightarrow{d} N(0,\sigma^2)$.
%\end{itemize}

In order to check conditions in previous lemma, we need the next auxiliary result, which is stated as follows

\begin{lem}
\label{lemC}
$\{D_n, \mathcal{F}_n, n \ge 1\}$ in \eqref{martdiff} are bounded martingale differences.
\end{lem}

\begin{proof}
First, we prove that $M_n$ in \eqref{mart} is a martingale:
\begin{equation}
\begin{array}{ll}
\E[M_{n+1}|\mathcal{F}_n] & = \dfrac{1}{a_{n+1}}\E[S_n+X_{n+1}-\E(S_n)-\E(X_{n+1})|\mathcal{F}_n] \\[0.4cm]
& =  \dfrac{1}{a_{n+1}}\big(S_n+\E[X_{n+1}|\mathcal{F}_n] -\E(S_n)-\E(X_{n+1})\big).
\end{array}
\end{equation}

By noticing that
\begin{equation}
\begin{array}{lll}
\E[X_{n+1}|\mathcal{F}_n] & = & \displaystyle\sum_{y \in \{0,1\}}\P(X_{n+1}=1|\mathcal{F}_n,Y_n=y)\P(Y_n=y) \\
                                         & = &   (1-\lambda_n)\alpha_n + \lambda_n\left(\alpha_n + \dfrac{\beta_n}{n}S_n\right) \\
                                         & = & \alpha_n + \dfrac{\beta_n}{n}\lambda_nS_n,
\end{array}
\end{equation}
and, given the independence of $Y_n$ and $S_n$
\begin{equation}
\begin{array}{ll}
\E(X_{n+1})  &=  \displaystyle\sum_{y \in \{0,1\}}\sum_{s=0}^n\P(X_{n+1}=1|S_n =s,Y_n=y)\P(S_n=s)\P(Y_n=y) \\[0.5cm]
                                          &=    (1-\lambda_n)\alpha_n + \lambda_n\big(\alpha_n \P(S_n=0)  + \cdots + (\alpha_n + \beta_n)\P(S_n=n)\big) \\[0.2cm]
                                          &=  \alpha_n + \dfrac{\beta_n}{n}\lambda_n\displaystyle\sum_{i=0}^n i \P(S_n=i) =  \alpha_n + \dfrac{\beta_n}{n}\lambda_n\E(S_n). 
                                          
\end{array}
\end{equation}

Then $\{M_n\}_{n \in \N}$ is a martingale. Now, observe that for all $n\ge 2$
\begin{equation}
\label{Dn}
D_n = \dfrac{X_n - \E(X_n)}{a_n} - \dfrac{S_{n-1} - \E(S_{n-1})}{n-1}\dfrac{\beta_{n-1}}{a_n}\lambda_{n-1},
\end{equation}
and 
\begin{equation}
\label{bounded}
|D_n| \le \dfrac{2}{a_n}, \ \ \ n\ge 1.
\end{equation}
In other words, $\{D_n, \mathcal{F}_n, n \ge 1\}$ are bounded martingale differences.
\end{proof}

\begin{proof}[Proof of Theorem \ref{teo1}]
First of all, note that
\begin{equation}
\label{lem1}
\dfrac{a_n}{n} = \dfrac{1}{n}\prod_{k=1}^{n-1}\left(\dfrac{k+\beta_k\lambda_k}{k} \right) = \prod_{k=1}^{n-1}\left( \dfrac{k+\beta_k\lambda_k}{k+1} \right),
\end{equation}
since $\beta_k\lambda_k \le 1$, thus $a_n/n$ is non-increasing. Moreover, $\frac{a_n}{n} = \exp(\sum_{k=1}^{n-1}\log (\frac{k+\beta_k\lambda_k}{k+1}) ) \break = \exp (- \sum_{k=1}^{n-1}\frac{1-\beta_k\lambda_k}{k+1} + O(1))$. Then, $\lim_{n\to\infty} a_n/n = 0$ if and only if $\sum_{k=1}^{\infty} \frac{1-\beta_k\lambda_k}{k+1} = \infty$

Initially assume \eqref{cond1}. That implies $\frac{a_n}{n} \to 0$. Given Lemma \ref{lemC}, define the martingale differences $Z_i = \frac{a_i}{i}D_i$. By \eqref{bounded} 
$\sum_{i=1}^{\infty} \E[Z_i^2| \mathcal{F}_{i-1}] \le \sum_{i=1}^{\infty} \frac{4}{i^2} < \infty$ a.s. 
Lemma \ref{lemA} implies that $\sum_{i=1}^{\infty} Z_i$ converges a.s. Now we use Kronecker's lemma to obtain \eqref{LLN}, that is
\begin{equation}
\frac{a_n}{n}M_n = \frac{a_n}{n} \sum_{i=1}^{n} D_i \to 0 \, \text{ a.s. as } n \to \infty.
\end{equation}
On the other hand, suppose $\sum_{k=1}^{\infty} \frac{1-\beta_k\lambda_k}{1+k} < \infty$. Apply 
\eqref{lem1} to say $v=\lim_{n\to\infty} \frac{a_n}{n} \in (0.1]$ and using \eqref{bounded}, 
$\sum_{i=1}^{\infty} \E[D_i^2| \mathcal{F}_{i-1}] \le \sum_{i=1}^{\infty} \frac{4}{a_i^2} < \infty$ a.s. Again, 
from Lemma \ref{lemA} we obtain $M_n = \sum_{i=1}^{\infty} D_i$ converges a.s. to some random variable $M$, with 
$Var(M)=\lim_{n\to\infty} Var(M_n) = \sum_{i=1}^{\infty} \E(D_i^2) > 0$. Hence, $M$ is a non-degenerate random 
variable. Thus, $\frac{a_n}{n} M_n$ also converges to a non-degenerate random variable, and that completes the proof.
\end{proof}

\begin{proof}[Proof of Theorem \ref{teo2}]
This proof will mainly follow the scheme presented in \cite{ZZ}. Proving item $(a)$, the Skorohod embedding theorem allows us to redefine $\{X_n,\mathcal{F}_n\}$ in a new probability space such that there is a Brownian motion $\{W(t)\}$ and a $\mathcal{F}_n -$filtered sequence of random variables $\tau_n \geq 0$ such that $M_n \overset{d}{=} W(T_n)$ (we can assume that $M_n = W(T_n)$ without loss of generality), where $T_n = \sum_{i=1}^{n} \tau_i$. Furthermore 
$$
\E[\tau_j|\mathcal{F}_{j-1}] = \E[D_j^2|\mathcal{F}_{j-1}] \ \ \ \ \mbox{and also} \ \ \ \ \E[\tau_j^p|\mathcal{F}_{j-1}] = \E[D_j^{2p}|\mathcal{F}_{j-1}] \ \ \mbox{a.s.}
$$

Since $n^{-1}(S_n - \E(S_n)) \xrightarrow{a.s.} 0$, we get that 
$$
\E[D_j^2|\mathcal{F}_{j-1}] = \frac{p_j(1-p_j)}{a_j^2} + o_{a.s.}\left(\frac{1}{a_j^2}\right) \ ,
$$
which implies that
$$
\sum_{j=1}^n \E[\tau_j|\mathcal{F}_{j-1}] = B_n^2 + o_{a.s.}(A_n^2) = B_n^2 + o_{a.s.}(B_n^2) \ .
$$
Here we recall that by ``$y_n$ is $o_{a.s.}(x_n)$" we mean $x_n/y_n \xrightarrow{a.s.} 0 $.
By noticing that $\hat{\tau}_j :\overset{d}{=} \tau_j - \E(\tau_j|\mathcal{F}_{j-1})$ is a sequence of martingale differences with respect to $\mathcal{F}_j$, and since $|D_j| \leq 2/a_j$, we conclude that $\E(\hat{\tau}_j|\mathcal{F}_{j-1}) \leq C a_j^{-4}$, for some positive constant $C$. This implies that $A_j^{-4}\E(\hat{\tau}_j^2|\mathcal{F}_{j-1})$ is summable.

Now we combine Theorem $2.18$ from \cite{H&H} with Kronecker's lemma to conclude that $\sum_{j=1}^{n}\hat{\tau}_j = o_{a.s.}(A_n^2) = o_{a.s.}(B_n^2)$, which in turn implies that $T_n=B_n^2 + o_{a.s.}(B_n^2)$. Then we apply Theorem $1.2.1$ of \cite{CR} combined with the fact that $B_n$ diverges to obtain
$$
M_n = W(T_n) = W(T_n^2) + o_{a.s.}(B_n\sqrt{\log \log B_n}) \ .
$$
Finally we divide both sides of the last equality above by $B_n \sqrt{\log \log B_n}$, and let $n$ goes to infinity to conclude the proof.

For item $(b)$, the goal is to prove that for any $\epsilon>0$ the probability $\P\left(\frac{|W(T_n) - W(B_n^2)|}{B_n}>\epsilon\right)$ vanishes as $n$ diverges. For any $\delta \in (0,1)$ we can decompose the above probability to get

$
\P\left(\frac{|W(T_n) - W(B_n^2)|}{B_n}>\epsilon, \frac{|T_n - B_n^2|}{B_n^2} > \delta\right) 
                                                    + \P\left(\frac{|W(T_n) - W(B_n^2)|}{B_n}>\epsilon, \frac{|T_n - B_n^2|}{B_n^2} \leq \delta\right) \ .
$

It is straightforward to see that the first term above vanishes as $n \to \infty.$ The second one can be bounded above by $\P(\limsup_{|s-1|\leq \delta}|W(s) - W(1)|>\epsilon)$, which also goes to zero as $\delta \to 0$ by the L\'evy modulus of continuity for Wiener processes. For a small enough $\delta$, let $n \to \infty$ to conclude the proof. 
\end{proof}

\begin{proof}[Proof of Theorem \ref{teo3}]
The strategy for this proof is to link the model to a generalized P\'olya urn problem, and then use the results stated by Svante Janson in \cite{Janson}.
%For a more general link between BSRD and P\'olya urn models, we refer the reader to \cite{GL}.

The first step consists in relate the distribution of $\{X_n\}_{n \in \mathbb{N}}$ with the distribution of the red balls in a two-color P\'olya urn (namely $\{R_n\}_{n \in \mathbb{N}}$). 
First of all, let us construct the \emph{random replacement matrix} for the generalized P\'olya urn. Here we will follow the notation given in \cite{Janson}. For this, consider the two column replacement vectors $\xi_{1}=(\xi_{11}, \xi_{12})$ (red) and $\xi_{2}=(\xi_{11}, \xi_{12})$ (blue), with $\xi_i \in \{(0,1),(1,0)\}$ (a single ball is replaced at each time), and the random replacement matrix given by $M = (\xi_1;\xi_2)$. Then if we chose replacement vector $\xi_1$ to reinforce the urn, it means that we will replace $\xi_{11}$ red balls and $\xi_{12}$ blue balls. Otherwise  we choose vector $\xi_2$, and then replace $\xi_{21}$ red, and $\xi_{22}$ blue balls.

At each step, a ball is drawn from the urn, its color is observed, and the ball is replaced. The replacement column vector is then chosen according to the color of the withdrawn ball.

Note that $\P(\xi_{ij}=1) = \E(\xi_{ij})$, for all $i,j$. If at time $n$ we get $r$ red balls ($R_n=r$), then the probabilities of replacing a red ball (success) or a blue ball (failure) at time $n+1$ (conditioned on a proportion $r/T_n$ of red balls) are, respectively, given by

%At each step, a ball is drawn from the urn, its color is observed, and the ball is replaced. Player $\mathcal{A}$ tosses a coin, and decides if the replacement vector will be the same as the color of the ball, or the opposite. Player $\mathcal{B}$ behaves simpler: He (She) only tosses a coin in order to chose the replacement vector. At each step, player $\mathcal{A}$ is chosen with probability $\lambda$. Note that $\P(\xi_{ij}=1) = \E(\xi_{ij})$, for all $i,j$. For this, notice that if at time $n$ we get $r$ red balls ($R_n=r$), then the probabilities of replacing a red ball (success) or a blue ball (failure) at time $n+1$ are, respectively, given by
\begin{equation}\label{red}
\begin{array}{ll}
\P(R_{n+1}=r+1 | R_n=r) &= \P(\xi_{11}=1)\dfrac{r}{T_n} + \P(\xi_{21}=1)\left(1-\dfrac{r}{T_n}\right) \\
&= \E[\xi_{21}] + \left(\E[\xi_{11}]-\E[\xi_{21}]\right)\dfrac{r}{T_n},
\end{array}
\end{equation}
and
\begin{equation}\label{blue}
\P(R_{n+1}=r| R_n=r) = \E[\xi_{22}] + \left(\E[\xi_{12}]-\E[\xi_{22}]\right)\frac{r}{T_n} \ ,
\end{equation}
where $T_n=R_n+B_n$ is the total number of balls at time $n$ ($B_n$ being the number of blue balls).
Now we obtain the transition probabilities for $\{X_n\}$. By the independence of $Y_n$ and $S_n$, we can repeat the argument used in the proof of Lemma \ref{lemC} (sum in all sets $\{Y_n=y\}$) to obtain from \eqref{epi} the following probabilities (conditioned on a proportion $s/n$ of previous successes)
\begin{equation}\label{suc}
\P(X_{n+1}=1 | S_n=s) = \alpha_0-\alpha_0\beta + \lambda\beta\frac{s}{n}
\end{equation} 
\begin{equation}\label{fai}
\P(X_{n+1}=0 | S_n=s) = 1- \alpha_0+\alpha_0\beta - \lambda\beta\frac{s}{n}
\end{equation}

The next step is to construct a matrix $A$ given by
$$
A= \E(M) = \left(
\begin{array}{cc}
\E(\xi_{11}) & \E(\xi_{21}) \\
\E(\xi_{12}) & \E(\xi_{22})
\end{array}
\right) 
$$
Then we relate \eqref{red} with \eqref{suc} and \eqref{blue} with \eqref{fai} to obtain the expected replacement matrix
$$
A=  \left(
\begin{array}{cc}
\alpha_0+\beta(\lambda-\alpha_0) & \alpha_0-\beta\alpha_0 \\
 1-\alpha_0-\beta(\lambda-\alpha_0) & 1- \alpha_0+\beta\alpha_0
\end{array}
\right)
$$
%Now we obtain the key quantities to the limiting theorems stated in \cite{Janson}. As is known by that paper, the limiting theorems depend on the eigendecomposition of $A$. The two eigenvalues of $A$ are $\ell_1 = 1$ and $\ell_2 = \beta \lambda$, and $v_1 = \frac{1}{1-\beta \lambda}{{\alpha_0-\beta \alpha_0}\choose{1-\alpha_0+\beta_0}}$ is the eigenvector associated to $\ell_1$. 
Now we obtain the key quantities to the limiting theorems stated in \cite{Janson}. As is known by that paper, the limiting theorems depend on the eigendecomposition of $A$. The two eigenvalues of $A$ are $\ell_1 = 1$ and $\ell_2 = \beta \lambda$, and $v_1 = \frac{1}{1-\beta \lambda}{{\alpha_0-\beta \alpha_0}\choose{1-\alpha_0 - \beta(\lambda - \alpha_0)}}$ is the eigenvector associated to $\ell_1$. 

Now all the calculations are done by using the results stated in \cite{Janson} (and examples therein). For item $(i)$ we apply Theorem $3.22$ of \cite{Janson} combined with Lemmas $5.4$ and $5.3(i)$ of the same paper to obtain \eqref{tcl1} and \eqref{var1}. For item $(ii)$ we apply Theorem $3.23$ of the same paper to obtain \eqref{tcl2} and \eqref{var2}.

These last two arguments conclude the proof.

\end{proof}
\bigskip

\subsection*{Acknowledgements}
The authors thank G. Ludwig, T. Vargas, M. Abadi and G. Ost for several comments. MGN was supported by Funda\c{c}\~ao de Amparo \`a Pesquisa do Estado de S\~ao Paulo, FAPESP (grant 2015/02801-6). He thanks kind hospitality and financial support from FAMAT-UFU. RL is partially supported by FAPESP (grant 2014/19805-1), and CNPQ PDJ grant (process 406324/2017-4). He thanks Center for Neuromathematics (IME-USP) for warm hospitality.

\end{document}